\documentclass[12pt]{amsart}
\usepackage[utf8]{inputenc}% 
\usepackage[T1]{fontenc}%
\usepackage{lmodern}% To get high quality fonts
\usepackage[top=1.15in, bottom=1.15in, left=1.2in, right=1.2in]{geometry}
\usepackage{amsmath,amssymb}
\usepackage{amsthm}
\usepackage{mathtools}
\usepackage{mathrsfs} % renders \mathscr cmd
\usepackage{xfrac} % renders diagonal frac notation: use \sfrac{}{}
\usepackage{dsfont} % renders '1' for characteristic function...
\usepackage{array}
\usepackage{verbatim}
\usepackage{graphicx}
\usepackage{enumitem} % Give extra customization on top of itemize and enumerate
\usepackage{lipsum}% 
\usepackage{relsize}
\usepackage{exscale}
\usepackage{tikz-cd}%https://tikzcd.yichuanshen.de/
\usepackage[hidelinks]{hyperref}
\theoremstyle{plain} %% This is the default, anyway
\begingroup % Confine the \theorembodyfont command
\newtheorem{thm}{Theorem}[section] % Numbered within each section
\newtheorem{cor}[thm]{Corollary} % Numbered along with thm
\newtheorem{lem}[thm]{Lemma} % Numbered along with thm
\newtheorem{prop}[thm]{Proposition} % Numbered along with thm
\endgroup

\theoremstyle{definition}
\newtheorem{definition}[thm]{Definition} % Numbered along with thm
\theoremstyle{remark}
\newtheorem{remark}[thm]{Remark} % Numbered along with thm
\newtheoremstyle{ser}% name
{8pt}% Space above
{8pt}% Space below
{\it}% Body font
{}% Indent amount
{\sf\bfseries}% Theorem head font
{:}% Punctuation after theorem head
{6mm}% Space after theorem head
{}% Theorem head spec (can be left empty, meaning `normal')

\newtheoremstyle{serr}% name
{8pt}% Space above
{8pt}% Space below
{\normalfont}% Body font
{}% Indent amount
{\sf}% Theorem head font
{.}% Punctuation after theorem head
{6mm}% Space after theorem head
{}% Theorem head spec (can be left empty, meaning `normal')

\theoremstyle{ser}

\theoremstyle{serr}

\theoremstyle{ser}

\numberwithin{equation}{section}

     \newcommand{\sB}{\mathcal B}		\newcommand{\kB}{\mathscr{B}}

     		\newcommand{\kH}{\mathscr{H}}

     		\newcommand{\kM}{\mathscr{M}}

\def\XXint#1#2#3{{\setbox0=\hbox{$#1{#2#3}{\int}$ }
		\vcenter{\hbox{$#2#3$ }}\kern-.6\wd0}}

\newlength{\mylen}
\setbox1=\hbox{$\bullet$}\setbox2=\hbox{\tiny$\bullet$}
\newcommand{\R}{\mathbb{R}}
\newcommand{\Z}{\mathbb{Z}}
\newcommand{\C}{\mathbb{C}}

\newcommand{\N}{\mathbb{N}}

\newcommand{\D}{\mathbb{D}}

\newcommand{\ep}{\varepsilon}

\newcommand{\ud}{\,\mathrm{d}} 
\newcommand{\norm}[1]{\| #1 \|}

\newcommand{\ran}{\textnormal{\textrm{ran}}}
\newcommand{\kernel}{\textnormal{\textrm{ker}~}}

\newcommand{\sot}{\mathrm{SOT}}
\newcommand{\wot}{\mathrm{WOT}}

\newcommand{\rP}[1]{\mathbf{R}\!\left( #1 \right)}
\newcommand{\spec}[1]{\textnormal{sp}\left( #1 \right)}

\usepackage[utf8]{inputenc}
\usepackage{graphicx} % Required for inserting images
\usepackage{verbatim}

\newcommand{\Title}{On the convergence of the normalized power sequence of spectral operators on Hilbert space}%
\newcommand{\ShortTitle}{}

\makeatletter
\newcommand{\raisemath}[1]{\mathpalette{\raisem@th{#1}}}
\newcommand{\raisem@th}[3]{\raisebox{#1}{$#2#3$}}
\makeatother

\newcommand{\AuthorOne}{Soumyashant Nayak}%
\newcommand{\AuthorOneAddr}{%
Statistics and Mathematics Unit, Indian Statistical Institute, 8th Mile, Mysore Road, RVCE Post, Bengaluru and Karnataka - 560 059, India
}%
\newcommand{\AuthorOneEmail}{%
soumyashant@isibang.ac.in
}%

\newcommand{\AuthorTwo}{Renu Shekhawat}%
\newcommand{\AuthorTwoAddr}{%
Statistics and Mathematics Unit, Indian Statistical Institute, 8th Mile, Mysore Road, RVCE Post, Bengaluru and Karnataka - 560 059, India
}%
\newcommand{\AuthorTwoEmail}{%
rs\_math1904@isibang.ac.in
}%

\newcommand{\Keywords}{Spectral radius formula, Yamamoto's theorem, Spectral operators, Haagerup-Schultz theorem}

\newcommand{\SubjectClassification}{47B40, 47A10, 15A18}%

\title[\MakeUppercase\ShortTitle]{\MakeUppercase \Title}
\author{\AuthorOne}%
\address[\AuthorOne]{\AuthorOneAddr}%
\email{\href{mailto:\AuthorOneEmail}{\AuthorOneEmail}}%

\author{\AuthorTwo}%
\address[\AuthorTwo]{\AuthorTwoAddr}%
\email{\href{mailto:\AuthorTwoEmail}{\AuthorTwoEmail}}%

\date{}%
\begin{document}

\subjclass[2020]{\SubjectClassification}%

\begin{abstract}
Let $\kH$ be a complex Hilbert space, and let $\kB(\kH)$ denote the set of all bounded operators on $\kH$. For an operator $T \in \kB(\kH)$, let $|T| := (T^*T)^{\frac{1}{2}}$. For $A \in \kB(\kH)$, we refer to the sequence, $\{ |A^n|^{\frac{1}{n}} \}_{n \in \N}$, as the {\it normalized power sequence} of $A$. As our main result, we prove that the normalized power sequence of a spectral operator in $\kB(\kH)$ converges in norm, and provide an explicit description of the limit in terms of its idempotent-valued spectral resolution.\ Our approach substantially generalizes the corresponding result by the first-named author in the case of matrices in $M_m(\C)$ and supplements the Haagerup-Schultz theorem on $\sot$-convergence of the normalized power sequence of an operator in a $II_1$ factor.
\end{abstract}

%-------------------------------------------------------------------%

\maketitle

%-------------------------------------------------------------------%
\section{Introduction}
%-------------------------------------------------------------------%

\vskip 0.1in

Let $\kH$ be a complex Hilbert space, and let $\kB(\kH)$ denote the set of all bounded linear operators on $\kH$. For $T \in \kB(\kH)$, the spectral radius formula, $$r(T) = \lim_{n \to \infty} \|T^n\|^{\frac{1}{n}},$$ indicates that the asymptotic growth-rate of the sequence, $\left\{ \|T^n\| \right\}_{n \in \N}$, is governed by, and thus carries information about, the spectral radius of $T$. This suggests that the asymptotic behaviour of the power sequence, $\{ T^n \}_{n \in \N}$, may yield finer information about the spectral properties of $T$. In this vein, the spectral radius formula was considerably generalized by Yamamoto as follows.

\vskip 0.1in

\noindent {\bf Yamamoto's theorem } (see \cite[Theorem 1]{yamamoto})
\textsl{Let $A$ be a matrix in $ M_m(\C)$ and $|\lambda_j|(A)$ denote the $j^{\textrm{th}}$-largest number in the list of modulus of eigenvalues of $A$ (counted with multiplicity). Then $$\lim_{n \to \infty} s_j(A^n)^{\frac{1}{n}} = |\lambda_j|(A),$$ where $s_j(A^n)$ denotes the $j^{\textrm{th}}$-largest singular value of $A^n$.
}

\vskip 0.1in

In \cite[Theorem 3]{yamamotoT}, Yamamoto extended the above theorem to a class of compact operators on an infinite-dimensional complex Hilbert space satisfying certain technical conditions. Subsequently, building on Yamamoto's approach, Davis \cite{Davis} later achieved a full extension of the theorem to all compact operators on an infinite-dimensional complex Hilbert space.

For an operator $T \in \kB(\kH)$, let $|T| := (T^* T)^{\frac{1}{2}}$. For $A \in \kB(\kH)$, we refer to the sequence $\{|A^n|^\frac{1}{n}\}_{n \in \N}$,  as the normalized power sequence of $A$. In \cite{stronger_form_yamamoto}, the first-named author has shown a spatial version of Yamamoto's theorem which asserts that the normalized power sequence of any matrix in $M_m(\C)$ converges in norm. In fact, an explicit description of the limiting positive operator is provided in terms of the diagonalizable part in the Jordan-Chevalley decomposition of the matrix as follows.

\vskip 0.1in

\noindent {\bf Theorem 3.8 in \cite{stronger_form_yamamoto}.} 
\textsl{
Let $A \in M_m(\C)$ and $\{ a_1, \ldots, a_k \}$ be the set of modulus of eigenvalues of $A$ such that $0 \le a_1 < a_2 < \cdots < a_k$. Let $A = D + N$ be the Jordan-Chevalley decomposition of $A$ into its commuting diagonalizable and nilpotent parts ($D, N$, respectively). For $1 \le j \le k$, let $E_j$ be the orthogonal projection onto the subspace of $\C ^m$ spanned by the eigenvectors of $D$ corresponding to eigenvalues with modulus less than or equal to $a_j$, and set $E_0 := 0$. Then the following assertions hold:
\begin{itemize}
    \item[(i)] The normalized power sequence of $A$ converges in norm to the positive-semidefinite matrix, $$\sum_{j=1}^k a_j (E_j - E_{j-1}).$$
    \item[(ii)] A non-zero vector $x \in \C ^m$ is in $\mathrm{ran}(E_j) \backslash \mathrm{ran}(E_{j-1})$ if and only if $$\lim_{n \to \infty} \|A^n x\|^{\frac{1}{n}} = a_j.$$ 
\end{itemize}
}

\vskip 0.1in

In \cite{huang_tam_extensions}, the above result has been generalized to the context of real semisimple Lie groups by Huang and Tam. In \cite{bhat_bala}, Bhat and Bala have shown that the normalized power sequence of a compact operator on an infinite-dimensional complex separable Hilbert space is norm-convergent, making essential use of the above result for matrices. In this article, as our main result we prove a generalization of \cite[Theorem 3.8]{stronger_form_yamamoto} to the context of spectral operators on $\kH$ (see \cite{dunford_schwartz_III}). We note the relevant result below.

\vskip 0.1in

\noindent {\bf The Main Result} (see Theorem \ref{thm:spectral}) 
\textsl{
Let $A$ be a spectral operator in $\kB(\kH)$, and $e_A$ be the idempotent-valued spectral resolution of $A$. For $\lambda \in \R$, let  $F_{\lambda}$ denote the orthogonal projection onto the range of $e_A(\D_{\lambda})$, where $\D_{\lambda} \subset \C$ denotes the closed disc of radius $\lambda$ centred at the origin. Then the following assertions hold:
\begin{itemize}
    \item[(i)]  $\{ F_{\lambda} \}_{\lambda \in \R}$ is a bounded resolution of the identity and 
    $$\lim_{n \to \infty} |A^n|^{\frac{1}{n}} = \int_{0}^{\mathrm{r}(A)} \lambda \; \ud F_{\lambda}, \textrm{ in norm}.$$
    Moreover, the spectrum of the limiting positive operator, $\int_{0}^{\mathrm{r}(A)} \lambda \; \ud F_{\lambda}$, is, $$|\spec{A}| := \{ |\lambda| : \lambda \in \spec{A} \},$$ the modulus of the spectrum of $A$.
    \item[(ii)]  For every vector $x \in \kH$, there is a smallest non-negative real number $\lambda_x$ such that $x$ lies in the range of the spectral idempotent $e_A(\D_{\lambda _x})$, which may be obtained as the following limit, $$\lim_{n \to \infty} \|A^n x\|^{\frac{1}{n}} = \lambda _x.$$
\end{itemize}
}

\vskip 0.1in

Furthermore, in Theorem \ref{thm:exp_one_p_group}, as a consequence of Theorem \ref{thm:spectral}, we obtain some results pertaining to the asymptotic behaviour of one-parameter groups in $\kB(\kH)$ whose infinitesimal generator is a spectral operator. Our techniques in proving Theorem \ref{thm:spectral} are necessarily different from those used in \cite{stronger_form_yamamoto}, where the trace in $M_m(\C)$ plays an important role in inferring spectral properties of $A$ via the computation of its moments. Since there are spectral operators which are \textbf{not} trace-class (for example, invertible scalar-type operators when $\kH$ is infinite-dimensional), a direct imitation of the proof strategy in \cite{stronger_form_yamamoto} is destined to fail.

On another note, it is worth pointing out that although in \cite{stronger_form_yamamoto} the spectral radius formula is never used, and in fact, follows as a consequence of \cite[Theorem 3.8]{stronger_form_yamamoto}, in this article we make generous use of the spectral radius formula in our arguments. Fuglede's theorem (see \cite[Theorem I]{fuglede}) helps us in splitting the problem into two parts, one involving invertible spectral operators, and another involving spectral operators whose spectrum is contained in a small disc centred at the origin of $\C$. The argument is based on a careful analysis of the interplay of these parts.

Let $\kM$ be a $II_1$ factor.\ In \cite[Theorem 8.1]{haagerup_schultz}, Haagerup and Schultz showed that the normalized power sequence of an operator in $\kM$ converges in the strong-operator topology and described the spectral resolution of the limiting positive operator in terms of the so-called Haagerup-Schultz projections. An operator in $\kM$ is said to be $\sot$-quasinilpotent if its normalized power sequence converges to $0$ in the strong-operator topology. In this context, Dykema and Krishnaswamy-Usha (\cite{Angles_HS}) identified an appropriate modification of the notion of spectrality called the UNZA property (uniformly non-zero angles property). An operator in $\kM$ is said to have the UNZA property if the angles between its Haagerup–Schultz projections are uniformly bounded away from zero. In \cite[Theorem 4.7-(ii)]{Angles_HS}, it is noted that every operator in $\kM$ with the UNZA property has a decomposition of the form $D+Q$ where $D$ is a scalar-type operator and $Q$ is a $\sot$-quasinilpotent operator such that $DQ=QD$. In \cite[Example 5.1]{Angles_HS}, an example is given of an operator in $\kM$ which is {\bf not} spectral in the sense of Dunford.

Interestingly, the path in \cite{Angles_HS} is in the opposite direction to the one in this article. There, one starts with the Haagerup-Schultz theorem about $\sot$-convergence of the normalized power sequence of operators in $\kM$, and uses its consequences to establish the `spectrality' of a large class of operators in $\kM$ (the operators with UNZA property), whereas we start with the spectrality of an operator and use it to prove the norm-convergence of its normalized power sequence.

%-------------------------------------------------------------------%
\section{Preliminaries}
%------------------------------------------------------------------%

In this section, we document notation, and some basic concepts and lemmas needed to follow the discussion in this article. The main reference that we use for basic results on spectral operators is \cite{dunford_schwartz_III}.

\vspace{0.3cm}\noindent{}\textbf{Notations.}
Throughout this article, $\kH$ denotes a complex Hilbert space, and $\kB(\kH)$ denotes the set of bounded operators on $\kH$. Let $T$ be an operator in $\kB(\kH)$. We define $|T| := (T^*T)^{\frac{1}{2}}$. The spectrum of $T$ is denoted by $\spec T$, and the spectral radius of $T$ is denoted by $\mathrm{r}(T)$. We say $T$ is {invertible in $\kB(\kH)$} to mean that $T$ has a bounded inverse in $\kB(\kH)$, denoted by $T^{-1}$. The range of $T$ is denoted by $\ran (T)$, and the projection onto the closure of the range of $T$ is denoted by $\rP T$. The sequence $\{|T^n|^\frac{1}{n}\}_{n \in \N}$ is said to be the {\it normalized power sequence} of $T$. For $r \ge 0$, the closed disc of radius $r$ centred at the origin in $\C$, is denoted by $\D_r$, and for $r < 0$, we stipulate that $\D_{r} := \varnothing$. For a complex number $z$, we denote the real part of $z$ by $\Re z$, and the closed half-plane $\{z \in \C:\Re z \le r\}$ is denoted by $\mathbb{H}_r$ for all $r \in \R$.

\vskip 0.1in

We make frequent use of the $C^*$-identity, $\norm {T^*T} = \norm T^2 = \norm {TT^*}$. Below we list two basic operator inequalities labelled (OI1) and (OI2), which we invoke repeatedly in our discussion.

\vskip 0.1in
    
\noindent \textsl{(i)} (\cite[Corollary 4.2.7]{KR-I}) If $A$ and $B$ are self-adjoint operators in $\kB(\kH)$ such that $A \le B$, then
\begin{equation}
\label{eqn:OI3}
T^*AT ~\le~ T^*BT \quad \textrm{ for all }~ T \in \kB(\kH). \tag{OI1}    
\end{equation}

\vskip 0.1in
    
\noindent \textsl{(ii)} (\cite[Proposition V.1.9]{Bhatia}) If $H$ and $K$ are positive operators in $\kB(\kH)$ such that $0 \le H \le K$, then
\begin{equation}
\label{eqn:OI4}
0 ~\le~  H^\frac{1}{n} ~\le~ K^\frac{1}{n} \quad \textrm{ for all }~ n \in \N.
\tag{OI2}
\end{equation}

\vskip 0.1in

\noindent {\bf Fuglede's theorem} (see \cite[Theorem I]{fuglede}). Let $N \in \kB(\kH)$ be a normal operator. Then, for an operator $T \in \kB(\kH)$, $TN=NT$ if and only if $T$ commutes with every spectral projection of $N$. Consequently, $TN=NT$ if and only if $TN^* = N^*T$.

\begin{lem} 
\label{lem:isolated0}
\textsl{
Let $H$ be a positive operator in $\kB(\kH)$. Then the following are equivalent :
\begin{itemize}
    \item[(i)] $\ran (H)$ is a closed subspace of $\kH$.
    \item[(ii)] Either $H$ is invertible, or $0$ is an isolated point of $\spec{H}$.
    \item[(iii)] $\lim_{n \to \infty} H^{\frac{1}{n}} = \rP{H}$ in norm.
\end{itemize}
}
\end{lem}

\begin{proof}
Let $E := \rP{H}$. Since $H$ is self-adjoint, $H$ and $E$ commute. 

\vskip 0.05in

\noindent \textsl{Proof of (i)$\Longleftrightarrow$(ii)} : Note that $\ran(H) = \ran(E)$ if and only if $\ran(H)$ is a closed subspace of $\kH$. By the Douglas factorization lemma (see \cite{Douglas_lemma}), $\ran(H) = \ran(E)$ if and only if there are positive real numbers $0 < \lambda < \mu$ such that $\lambda ^2 E \le H^2 \le \mu ^2 E$. Since $(\cdot)^{\frac{1}{2}}$ is an operator-monotone function, we have $\lambda E \le H \le \mu E$. Using the continuous function calculus in the context of the commutative $C^*$-algebra generated by $H$ and $E$, we observe that this is equivalent to the spectrum of $H$ being contained in $\{0 \} \cup [\lambda, \mu]$.

\vskip 0.1in

\noindent \textsl{Proof of (ii)$\Longleftrightarrow$(iii)} : Note that the commutative $C^*$-algebra generated by $H$ may be viewed as $C(\spec{H})$, the space of complex-valued continuous functions on $\spec{H}$.  Let $f_n : \spec{H} \to \R_{\ge 0}$ denote the continuous function given by $x \mapsto x^{\frac{1}{n}}$. We observe that the sequence of continuous functions, $\{ f_n \}_{n \in \N}$, converges uniformly if and only if either $0 \notin \spec{H}$ or $0$ is an isolated point of $\spec{H}$; moreover, in that case, $f_n$ converges uniformly to the indicator function of $\spec{H} \setminus \{0\}$. At the level of operators, this is equivalent to the assertion that $\lim_{n \to \infty} H^{\frac{1}{n}} = \rP{H}$ in norm. 
\end{proof}

\begin{lem}
\label{lem:nth_root}
\textsl{
Let $H$ be a positive operator in $\kB(\kH)$, and $\alpha \ge 0$. Then $$(H^n+\alpha ^n I)^{\frac{1}{n}} \le H + \alpha I.$$
}
\end{lem}
\begin{proof}
Since $H^n + \alpha ^n I \le H^n + \alpha ^n I + \sum_{k=1}^{n-1} \binom{n}{k} \alpha ^{n-k} H^k = (H+\alpha I)^n$, the assertion follows from inequality (\ref{eqn:OI4}), or simply, the continuous function calculus.
\end{proof}

\begin{lem}
\label{lem:ran_order_pres}
\textsl{
Let $H, K$ be positive operators in $\kB(\kH)$ such that $0 \le H \le K$. Then $\rP{H} \le \rP{K}$.
}
\end{lem}

\begin{proof}
From inequality (\ref{eqn:OI4}), we observe that $0 \le H^{\frac{1}{n}} \le K^{\frac{1}{n}}$. Using \cite[Lemma 5.1.5]{KR-I} and taking $\sot$-limits as $n \to \infty$, we get the desired result.
\end{proof}

\begin{remark}
\label{rmrk:range_idem}
Let $T \in \kB (\kH)$, and let $S, S'$ be invertible operators in $\kB (\kH)$. Then $\ran (S T S') = \ran (ST) =  S \, \ran(T)$. Thus, $T$ has closed range if and only if $S T S'$ has closed range. In particular, for a projection $E \in \kB(\kH)$, it follows that $S^*ES$ has closed range, and 
$$\rP{S^*ES} = \rP{S^* E}.$$ 
\end{remark}

\begin{lem}\label{lem:rearrangment}
\textsl{
Let $E$ be a projection and $S$ be an invertible operator in $\kB(\kH)$. Then, 
$$\rP{S^*ES} = I - \rP{S^{-1}(I-E)S}.$$ 
}   
\end{lem}

\begin{proof}
Since $S^{-1}(I-E)S$ is a closed-range operator, and $S^*ES$ is a positive closed-range operator, it suffices to show that $ \ran (S^{-1}(I-E)S) = \kernel{S^*ES}$. Clearly, $(S^*ES)\left( S^{-1}(I-E)S \right) = 0$, whence it follows that
$$\ran (S^{-1}(I-E)S) \subseteq \kernel{S^*ES}.$$ 

Conversely, let $x \in \kernel{S^*ES} $. Since $S^*$ is invertible, $Sx \in \kernel{E}$. In other words, $ESx = 0$, whence $S^{-1}(I-E)Sx = x$, that is, $x \in \ran (S^{-1}(I-E)S)$. Thus, \[\kernel{S^*ES} \subseteq \ran (S^{-1}(I-E)S). \qedhere\]
\end{proof}

\begin{definition}[Resolution of the identity (see \cite{KR-I}, \S 5.2)]\label{def:resolution_of_the_identity}
A family $\{E_\lambda\}_{\lambda \in \R}$ of projections in $\kB(\kH)$ indexed by $\R$ is said to be a {\it resolution of the identity} on $\kH$ if it satisfies :
\begin{enumerate}
     \item[(i)] $\bigwedge_{\lambda \in \R} E_\lambda = 0$ and $\bigvee_{\lambda \in \R} E_\lambda = I$ ;
     \item[(ii)] $E_\lambda \le E_{\lambda'}$ when $\lambda \le \lambda'$ ;
     \item[(iii)] $E_\lambda = \bigwedge_{\lambda' > \lambda} E_{\lambda'}$.
 \end{enumerate}
 
If there is a constant $a \in \R$ such that $E_\lambda =0$ when $\lambda < -a$ and $E_\lambda = I$ when $a < \lambda$, then $\{E_\lambda\}_{\lambda \in \R}$ is said to be a {\it bounded} resolution of the identity.
\end{definition}

\begin{thm}[Spectral resolution of a self-adjoint operator. (see \cite{KR-I}, Theorem 5.2.2)]\label{thm:resolution_of_the_identity}
\label{thm:spec_res}
\textsl{
If $A$ is a self-adjoint operator in $\kB(\kH)$, then there is a resolution of the identity, $\{E_\lambda\}$, on $\kH$, such that
\begin{enumerate}
    \item[(i)] $E_\lambda = 0$ when $\lambda < -\|A\|$, and $E_\lambda = I$ when $\|A\| \le \lambda$ ;
    \item[(ii)] $AE_\lambda \le \lambda E_\lambda$ and $\lambda(I-E_\lambda) \le A(I-E_\lambda)$ for each $\lambda$ ;
    \item[(iii)] $A = \int_{-\|A\|}^{\|A\|} \lambda \ud{E_\lambda}$ in the sense of norm convergence of approximating Riemann sums.
\end{enumerate}
Moreover, $\lambda_0 \notin \spec{A}$ if and only if there is a neighborhood of $\lambda_0$ where $E_{\lambda}$, as a function of $\lambda$, is constant.
} 
\end{thm}
The family $\{E_\lambda\}_{\lambda \in \R}$ so determined, is called {\it the spectral resolution of $A$}.

\vskip 0.1 in

\begin{remark}
\label{rem:spectral_resolution_of_positive_operator}
\textsl{
Let $H$  be a positive operator in $\kB(\kH)$ with spectral resolution $\{E_\lambda\}_{\lambda \in \R}$. Then $E_\lambda = 0$ for $\lambda < 0$, and $E_\lambda = I$ for $ \lambda \ge \|H\|$, and $H = \int_{0}^{\|H\|} \lambda \ud{E_\lambda}$. For $\lambda_0 > 0$, the spectral resolution of the positive operator $(I-E_{\lambda_0})H$ is given by $\big\{ E_{\lambda_0} + (I-E_{\lambda_0}) E_{\lambda}\big\}_{\lambda \ge 0}$ .
}
\end{remark}

The notion of spectral resolution may be defined for a more general class of operators called {\it spectral operators}. Although spectral operators can be defined on Banach spaces, in this article we restrict attention to those acting on a Hilbert space. We refer to \cite{Dunford_survey} for the following definitions.

\begin{definition}[Idempotent-valued spectral resolution of an operator {(see \cite[\S 1]{Dunford_survey})}]
Let $T \in \kB(\kH)$ and $\sB$ be the $\sigma$-algebra of Borel sets in the complex plane. Let $e : \sB \to \kB(\kH) $ be an idempotent-valued map such that the following properties are satisfied : 
\begin{enumerate}
    \item [(i)] $T e(\Omega)=e(\Omega) T, \quad \spec{T|_{\ran (e(\Omega) ) }} \subseteq \overline{\Omega}, \quad \forall~\Omega \in \sB .$
    \item [(ii)] $e(\varnothing)= 0, \quad e(\C)=I, \;$ and $\; e \left(\C \setminus \Omega \right) = I - e(\Omega) \quad$ for $\Omega \in \sB$.
    \item[(iii)] For $\Omega _1, \Omega _2 \in \sB$, 
    \begin{align*}
        e(\Omega_1 \cap \Omega_2) &= e(\Omega_1) \wedge e(\Omega_2) = e(\Omega_1) e(\Omega_2),\\
        e(\Omega_1 \cup \Omega_2) &= e(\Omega_1) \vee e(\Omega_2) = e(\Omega_1) + e(\Omega_2) - e(\Omega_1) e(\Omega_2).
    \end{align*}
    \item [(iv)] $\norm{e(\Omega)} \leq M \; ~ \forall ~\Omega \in \sB$, for some constant $M \in \R_{\ge 0}$.
    \item [(v)] $e(\Omega)$ is countably additive in $\sot$, that is, for every sequence, $\left\{ \Omega_n \right\}_{n \in \N}$, of disjoint Borel sets, we have
$$
e\left(\bigcup_{n=1}^{\infty} \Omega_n \right) x = \sum_{n=1}^{\infty} e\left(\Omega_n \right) x, \quad \textrm{for every } x \in \kH.
$$
\end{enumerate}
If such a mapping $\Omega \mapsto e(\Omega)$ exists, then it is uniquely determined by $T$, and is called {\it the} idempotent-valued resolution of the identity for $T$, or {\it the} idempotent-valued spectral resolution of $T$. We denote the idempotent-valued spectral resolution of $T$ by $e_T$. 
\end{definition}

\begin{definition}[Spectral operator]
An operator $A \in \kB(\kH)$ which has an idempotent-valued spectral resolution is called a {\it spectral} operator.
\end{definition}

\begin{definition}[Scalar-type operator]
A spectral operator $D \in \kB(\kH)$ is said to be {\it scalar-type} if $D = \int_{\C} \lambda \, \ud e_D(\lambda)$, where $e_D$ is the idempotent-valued resolution of the identity for $D$.
\end{definition}

The following theorem gives the canonical reduction of bounded spectral operators, which we shall refer to as the Dunford decomposition of a spectral operator.

\begin{thm}[see \cite{dunford_schwartz_III}, Theorem XV.4.5]
\textsl{
An operator $A \in \kB(\kH)$ is spectral if and only if there is a scalar-type operator $D$ and a quasinilpotent operator $Q$, in $\kB(\kH)$ such that $DQ=QD$ and $A=D+Q$. Furthermore, this decomposition is unique, and $A$ and $D$ have identical spectra and identical idempotent-valued spectral resolutions.
}
\end{thm}

\begin{remark}\label{rem:scalar-type_are_diagonalizable}
The Dunford decomposition of a spectral operator may be viewed as a generalization of the Jordan-Chevalley decomposition of matrices, as scalar-type operators are similar to normal operators (see \cite[Theorem XV.6.4]{dunford_schwartz_III}). More explicitly, if $D \in \kB(\kH)$ is a scalar-type operator, there is a normal operator $N$ and an invertible operator $S$, in $\kB(\kH)$ such that $D = S^{-1}NS$. It follows that the spectral projection of $N$ corresponding to Borel set $\Omega$ is given by
$$E(\Omega) = S e_D(\Omega) S^{-1},$$
where $e_D$ denotes the idempotent-valued spectral resolution of $D$.
\end{remark}

\begin{thm}[see \cite{dunford_schwartz_III}, Theorem XV.5.6]\label{thm:analytic_functions_of_spectral_operators}
\textsl{
Let $A$ be a spectral operator in $\kB(\kH)$ and let $e_A$ be the idempotent-valued spectral resolution of $A$. If $f : \C \to \C$ is an analytic function, then $f(A)$ is a spectral operator whose idempotent-valued spectral resolution is given by
 $$e_{f(A)}(\Omega) = e_A\left(f^{-1}(\Omega)\right) \textrm{ for a Borel subset } \Omega.$$
}
\end{thm}

%-------------------------------------------------------------------%
\section{A Preparatory Theorem}
%-------------------------------------------------------------------%

For $S$ an invertible operator and $H$ a positive operator in $\kB(\kH)$, our main goal in this section is to prove the norm-convergence of the sequence, $\{ (S^*H^nS)^{\frac{1}{n}} \}_{n \in \N}$, as well as find its norm-limit.\ We first do this when $H$ has finite spectrum, and then use an approximation argument to obtain the result for the general case. This is a crucial step towards the proof of the main result of this article, namely Theorem \ref{thm:spectral}.

\begin{lem}
\label{lem:convergence_of_nth_roots_of_sp_sequence_of_positive_operators}
\textsl{
Let $H_1, H_2, \ldots, H_k$ be positive closed-range operators in $\kB(\kH)$. Then, for sequences $\{a_{1,n}\}_{n \in \N}, \{a_{2,n}\}_{n \in \N}, \ldots, \{a_{k,n}\}_{n \in \N}$ of non-negative real numbers, the limit
$$ \lim_{n \to \infty} \big(\sum_{i=1}^k a_{i,n}~ H_i\big)^\frac{1}{n} $$
exists in norm if and only if the limit
$$ \lim_{n \to \infty} \big(\sum_{i=1}^k a_{i,n}~ \rP{H_i}\big)^\frac{1}{n} $$
exists in norm. Moreover, when the limits exist, they coincide with each other.
}
\end{lem}

\begin{proof}
Let $h_{-1} : \R_{\ge 0} \to \R_{> 0}$ be the function defined by $h_{-1}(x) = x^{-1}$ for $x \in \R_{>0}$, and $h_{-1}(0) = 1$, and $h_{\frac{1}{2}} : \R_{\ge 0} \to \R_{\ge 0}$ denote the square root function, $x \mapsto \sqrt{x}$.

For each $1 \le i \le k$, since the range of $H_i$ is closed, by Lemma \ref{lem:isolated0}, $\spec{H_i} \setminus \{0\}$ is a compact set, whence $h_{-1}$ is an invertible continuous function on $\spec {H_i}$. Note that the restriction of the function $h_{\frac{1}{2}} h_{-1} h_{\frac{1}{2}}$ to $\spec{H_i}$ is the indicator function on $\spec{H_i}$ whose support is the closed set $\spec{H_i} \backslash \{ 0 \}$. Using the continuous function calculus for $H_i$, we observe that the operator $S_i := h_{-1}(H_i)$ is positive and invertible in $\kB(\kH)$, satisfying
$$H_i^{\frac{1}{2}}\, S_i\, H_i^{\frac{1}{2}} =h_{\frac{1}{2}}(H_i)\, h_{-1}(H_i)\, h_{\frac{1}{2}}(H_i) =  (h_{\frac{1}{2}} h_{-1} h_{\frac{1}{2}})(H_i) = \rP{H_i}.$$ 

Let $\alpha:= \min\limits_{1 \le i \le k}\{\|S_i^{-1}\|^{-1}\}$, and $\beta:= \max\limits_{1 \le i \le k}\{\|S_i\|\}$. Then, observe that
$$ \alpha I ~\le~ S_i ~\le~ \beta I ~\;~ ~;~ 1 \le i \le k.$$

It follows from (\ref{eqn:OI3}) that,
$$\alpha H_i ~\le~  H_i^\frac{1}{2}S_iH_i^\frac{1}{2} = \rP{H_i} ~\le~ \beta  H_i ~\;~ ~;~ 1 \le i \le k. $$

Since $a_{i,n}$'s are non-negative real numbers, for each $n \in \N$, we have
$$\alpha \sum_{i=1}^k a_{i,n}~ H_i ~\le~ \sum_{i=1}^ka_{i,n}~ \rP{H_i} ~\le~ \beta \sum_{i=1}^ka_{i,n}~ H_i.$$

It follows from (\ref{eqn:OI4}) that,
$$\alpha^\frac{1}{n}\big( \sum_{i=1}^k a_{i,n}~ H_i\big)^\frac{1}{n} ~\le~ \big(\sum_{i=1}^ka_{i,n}~ \rP{H_i} \big)^\frac{1}{n} ~\le~ \beta^\frac{1}{n} \big(\sum_{i=1}^ka_{i,n}~ H_i\big)^\frac{1}{n},$$

and by a simple algebraic manipulation,
$$\beta ^{-\frac{1}{n}} \big(\sum_{i=1}^ka_{i,n}~ \rP{H_i} \big)^\frac{1}{n} ~\le~ \big( \sum_{i=1}^k a_{i,n}~ H_i\big)^\frac{1}{n} ~\le~ \alpha^{-\frac{1}{n}} \big(\sum_{i=1}^ka_{i,n}~ \rP{H_i} \big)^\frac{1}{n}.$$

Since $0 < \alpha < \beta $, we have $\lim_{n \to \infty} \alpha ^{\frac{1}{n}} = \lim_{n \to \infty} \beta ^{\frac{1}{n}} = 1$, and the result follows from the sandwich theorem for limits.
\end{proof}

\begin{prop}
\label{thm:key1}
\textsl{
Let $a_1 < \cdots < a_k$ be non-negative real numbers and $H_1, \ldots, H_k$ be positive operators in $\kB(\kH)$ such that for each $ 1 \le i \le k$, the reverse cumulative sum, $\sum_{j=i}^k H_j$, is a closed-range positive operator. Define $G_i := \rP{\sum_{j=i}^k H_j}$ for $1 \le i \le k$, with the convention that $G_{k+1} := 0$. Then,
\begin{align*}
     \lim_{n \to \infty} \big(\sum_{i=1}^k a_i^n~ H_i\big)^\frac{1}{n}
     =& \sum_{i=1}^k a_i (G_i - G_{i+1}), \textrm{ in norm}.
\end{align*}
}
\end{prop}

\begin{proof}
%We first prove the norm-convergence, assuming that $\sum_{j=i}^k H_j$ are closed-range operators in $\kB(\kH)$.

%\vskip0.1in

\noindent For $1 \le i \le k$, let us denote the reverse cumulative sums by,
$$K_i := \sum_{j=i}^k H_j,$$
so that $G_i = \rP{K_i}$ for $1 \le i \le k$, and $G_{k+1} = 0$ as stipulated in the hypothesis of the theorem. Since $K_1 \ge K_2 \ge \cdots \ge K_k$, using Lemma \ref{lem:ran_order_pres}, we have 
$$G_1 \ge G_2 \ge \cdots \ge G_k .$$

Thus $\{G_i- G_{i+1} : 1 \le i \le k \}$ consists of mutually orthogonal projections, and using Abel's summation by parts, we have
\begin{align*}
     \Big( \sum_{i=1}^{k} a_i \big(G_i -  G_{i+1} \big)\Big)^n
   = ~\sum_{i=1}^{k} a_i^n~ \big(G_i - G_{i+1}\big) =~ a_1^n ~ G_1 + \sum_{i=1}^{k-1} ~(a_{i+1}^n - a_{i}^n) ~ G_{i+1},
\end{align*}

Hence, for all $n \in \N$, we have
\begin{equation}\label{eqn:lim_simplified}
    \Big( a_1^n ~ G_1 + \sum_{i=1}^{k-1} ~(a_{i+1}^n - a_{i}^n) ~ G_{i+1} \Big)^\frac{1}{n} = ~ \sum_{i=1}^{k} a_i \big(G_i -G_{i+1}\big).
\end{equation}

Again using Abel's summation by parts, we get
\begin{equation}
\label{eqn:abel_sum_for_sp_sequence_of_positive_operators}
\sum_{i=1}^k a_i^n~ H_i ~=~ a_1^n~ K_1 + \sum_{i=1}^{k-1}~ (a_{i+1}^n - a_i^n)~K_{i+1} ~.
\end{equation}

Since $\{a_1^n\}_{n \in \N}, \{a_2^n-a_1^n\}_{n \in \N}, \ldots, \{a_k^n-a_{k-1}^n\}_{n \in \N}$ are sequences of non-negative real numbers, and by our hypothesis $K_i$'s are positive closed-range operators, using Lemma \ref{lem:convergence_of_nth_roots_of_sp_sequence_of_positive_operators} and equation (\ref{eqn:lim_simplified}),
\begin{align*}
\lim_{n \to \infty} \big(\sum_{i=1}^k a_i^n~ H_i\big)^{\frac{1}{n}} &= \lim_{n \to \infty} \big(a_1^n~ K_1 + \sum_{i=1}^{k-1} ~(a_{i+1}^n - a_i^n)~K_{i+1}\big)^{\frac{1}{n}} \\
&= \lim_{n \to \infty} \big(a_1^n~ G_1 + \sum_{i=1}^{k-1} ~(a_{i+1}^n - a_i^n)~G_{i+1}\big)^{\frac{1}{n}} \\
& = \sum_{i=1}^{k} a_i \big(G_i -G_{i+1}\big),
\end{align*}
where the above limits are taken in the norm topology.
\end{proof}

\begin{cor}\label{cor:k_projections}
\textsl{
Let $a_1 < \ldots < a_k$ be non-negative real numbers. Let $E_1, \ldots, E_k \in \kB (\kH)$ be mutually orthogonal projections such that $E_1 + \cdots + E_k = I$, and $S$ be an invertible operator in $\kB(\kH)$. For $1 \le i \le k$, let  $F_i := \rP{S^{-1} (\sum_{j=1}^i E_j)S}$ with $F_0 := 0$.
then,
$$\lim_{n \to \infty} \big( \sum_{i=1}^k a_i ^n~ S^* E_i S \big)^{\frac{1}{n}}= \sum_{i=1}^k a_i (F_i - F_{i-1}), ~\;~ \textrm{ in norm.}$$
}
\end{cor}

\begin{proof}
Clearly for each $1 \le i \le k$, the operator $S^*E_i S$ is positive, and $\sum_{j=i}^k E_j$ is a projection, being a sum of mutually orthogonal projections. From Remark \ref{rmrk:range_idem}, the reverse cumulative sum $\sum_{j=i}^k S^*E_j S = S^*(\sum_{j=i}^k E_j)S$, is a positive closed-range operator.

Let $G_i := \rP{S^*(\sum_{j=i}^k E_j)S }$ for $1 \le i \le k$, with $G_{k+1} := 0$. From Lemma \ref{lem:rearrangment}, for $1 \le i \le k$ we observe that
$$ 
G_i =  I - \rP{S^{-1} \big( I - \sum_{j=i}^{k} E_j \big) S} = I - \rP{S^{-1}\big(\sum_{j=1}^{i-1} E_j\big)S} = I - F_{i-1}.
$$

Hence, keeping in mind that $F_k = I$, we have $G_{i} - G_{i+1} = F_{i} - F_{i-1}$, for each $1 \le i \le k$. The result then follows from Theorem \ref{thm:key1}.
\end{proof}

\begin{lem}\label{lem:spectral_resolution_of_the_limit}
\textsl{
Let $\{E_\lambda\}_{\lambda \in \R}$ be a bounded resolution of the identity on a Hilbert space $\kH$, and $S$ be an invertible operator in $\kB(\kH)$. Then the family $\big\{\rP{S^{-1} E_\lambda S}\big\}_{\lambda \in \R}$ defines a bounded resolution of the identity on $\kH$.
}
\end{lem}

\begin{proof}
Let $S_0$ be an invertible operator in $\kB(\kH)$. For $\lambda \in \R$, let $F_{\lambda} := \rP{S_0 ^* E_\lambda S_0}$. We first show that $\{F_{\lambda} \}_{\lambda \in \R}$ is a bounded resolution of the identity. From Remark \ref{rmrk:range_idem}, $F_{\lambda} = \rP{S_0^* E_{\lambda}}$. Since $S_0^*$ is invertible, note that $S_0^*~ \ran(E_{\lambda})$ is a closed subspace of $\kH$ and $\ran(F_{\lambda}) = S_0 ^* ~\ran (E_{\lambda})$.
\begin{enumerate}
    \item[(i)] Since $\bigcap_{\lambda \in \R} \ran(E_{\lambda}) = \{ 0 \}_{\kH}$ and $\bigcup_{\lambda \in \R} \ran(E_{\lambda})$ is dense in $\kH$, from the invertibility of $S_0 ^*$ we have
    \begin{align*}
    \bigcap_{\lambda \in \R} \ran (F_{\lambda}) &= S_0^*  \mathsmaller{\bigcap}_{\lambda \in \R} \ran (E_{\lambda}) \big) = \{ 0 \}_{\kH},\\
    \bigcup_{\lambda \in \R} \ran(F_{\lambda}) &= S_0^* \big( \mathsmaller{\bigcup}_{\lambda \in \R} \ran(E_{\lambda}) \big)  \textrm{ is dense in } \kH.\end{align*}
    Thus, $\bigwedge_{\lambda \in \R} F_{\lambda} ~=~ 0$, and $\bigvee_{\lambda \in \R} F_\lambda = I$.
    
    \item[(ii)] If $\lambda \le \lambda'$, then \begin{align*}
    E_\lambda ~\le~ E_{\lambda'} & \implies S_0^* E_\lambda S_0 ~\le~ S_0^* E_{\lambda'}S_0 &\textrm{ by inequality (\ref{eqn:OI3})}\\
    &\implies \rP{S_0 ^* E_\lambda S_0} ~\le~  \rP{S_0 ^* E_{\lambda'}S_0} &\textrm{ by Lemma \ref{lem:ran_order_pres}}\\
    &\implies F_{\lambda} ~\le~ F_{\lambda '}.
    \end{align*}
    
    \item[(iii)] Since $E_{\lambda} = \bigwedge_{\lambda' > \lambda} E_{\lambda'}$, we observe that $\ran(E_{\lambda}) = \bigcap_{\lambda ' > \lambda} \ran(E_{\lambda '})$. Thus,
    $$\bigcap_{\lambda ' > \lambda } \ran(F_{\lambda '}) = S_0^* \big(\mathsmaller{\bigcap}_{\lambda ' > \lambda } \ran(E_{\lambda'})\big) = S_0^* ~ \ran(E_{\lambda}) = \ran(F_{\lambda}),$$
    which implies that $\bigwedge_{\lambda ' > \lambda} F_{\lambda'} = F_{\lambda}$.
\end{enumerate}
From the above, we conclude that $\{ F_{\lambda} \}_{\lambda \in \R}$ is a resolution of the identity. Note that $F_\lambda = 0, F_{\lambda} = I$, respectively, if and only if $E_\lambda = 0, E_{\lambda} = I$, respectively. Thus $\{F_\lambda\}_{\lambda \in \R}$ is a bounded resolution of the identity.

Let us choose $S_0$ to be the invertible operator $(S^{-1})^*$. Using Remark \ref{rmrk:range_idem}, we have
$$ \rP{S^{-1} E_{\lambda} S} = \rP{S^{-1} E_\lambda } = \rP{S_0^* E_{\lambda} }  = \rP{S_0 ^* E_{\lambda} S_0}.$$ 
From the preceding discussion, we conclude that $\big\{\rP{S^{-1} E_\lambda S}\big\}_{\lambda \in \R}$ defines a bounded resolution of the identity on $\kH$.
\end{proof}

\begin{thm}
\label{thm:application_of_spectral_theorem}
\textsl{
Let $H$ be a positive operator in $\kB(\kH)$ with spectral resolution $\{ E_{\lambda} \}_{\lambda \in \R}$, and $S$ be an invertible operator in $\kB(\kH)$. For $\lambda \in \R$, let $F_{\lambda} := R (S^{-1}E_{\lambda}S).$ Then $\{ F_{\lambda} \}_{\lambda \in \R}$ is a bounded resolution of the identity and
$$\lim_{n \to \infty} (S^* H^n S)^{\frac{1}{n}} = \int_{0}^{\|H\|} \lambda \; \mathrm{d}F_{\lambda},\textrm{ in norm}.$$
In addition, the spectrum of the limiting positive operator, $\int_{0}^{\|H\|} \lambda \; \mathrm{d}F_{\lambda}$, coincides with the spectrum of $H$.
} 
\end{thm}

\begin{proof}
From Lemma \ref{lem:spectral_resolution_of_the_limit}, $\{F_\lambda\}_{\lambda \in \R}$ is a bounded resolution of the identity. Since $F_\lambda - F_{\lambda'} = 0$ if and only if $E_\lambda - E_{\lambda'} = 0$, from Theorem \ref{thm:spec_res} it follows that the spectrum of the positive operator $\int_{0}^{\|H\|} \lambda \; \mathrm{d}F_{\lambda}$, coincides with the spectrum of $H$.

Without loss of generality (by replacing $H$ with $H/\|H\|$ if necessary), we may assume that $\|H\|=1$, so that $\spec{H} \subseteq [0, 1]$. Below we set up some notation. Let $K := \int_{0}^{\|H\|} \lambda \; dF_{\lambda}$. For $m \in \N$, we partition $[0,1]$ uniformly into intervals $\left[\frac{i-1}{m}, \frac{i}{m}\right] ~;~ 1 \le i \le m$, and let $E_{i,m}$ and $F_{i,m}$ denote the spectral projections for $H$ and $K$, respectively, corresponding to the interval $\left[0, \frac{i}{m}\right]$. The lower and upper approximations to $H$ are defined as follows :
$$H_m := \sum_{i=1}^{m} \tfrac{i-1}{m} (E_{i,m} - E_{i-1,m})~,~\;~ H'_m := \sum_{i=1}^{m} \tfrac{i}{m} (E_{i,m} - E_{i-1,m}).$$
\noindent Similarly, the lower and upper approximations to $K$ are defined as follows :
$$K_{m} := \sum_{i=1}^{m} \tfrac{i-1}{m} \left( F_{i,m} - F_{i-1,m} \right)~,~\;~ K_m' := \sum_{i=1}^{m} \tfrac{i}{m} \left( F_{i,m} - F_{i-1,m} \right).$$

Our main interest is in the norm-convergence of the sequence, $T_n := (S^*H^nS)^{\frac{1}{n}}$. The approximation argument involves lower and upper approximations of $T_n$ by 
$$T_{m,n} := (S^*H_m^nS)^{\frac{1}{n}}~,~\;~ T'_{m,n} := (S^*{H'_m}^nS)^{\frac{1}{n}}, \textrm{ for } m,n \in \N.$$

By Corollary \ref{cor:k_projections}, for every $m \in \N$, we have 
$$\lim_{n \to \infty} T_{m,n} = K_m ~,~\;~ \lim_{n \to \infty} T'_{m,n} = K'_m .$$

Given $\varepsilon > 0$, fix a positive integer $m_{\varepsilon}$ such that $\frac{1}{m_{\varepsilon}} < \frac{\varepsilon}{2}$. There is an $n(\ep) \in \N$ such for all $n \ge n(\ep)$, we have 
\begin{align*}
\|T_{m_{\varepsilon}, n} - K_{m_{\varepsilon}}\| < \frac{\varepsilon}{2}~,~\;~ \|T'_{m_{\varepsilon}, n} - K'_{m_{\varepsilon}}\| < \frac{\varepsilon}{2}.
\end{align*}

It follows that,
\begin{align}\label{eqn:appx1}
-\tfrac{\varepsilon}{2} I ~\le~ T_{m_{\varepsilon}, n} - K_{m_{\varepsilon}} ~\le~  \tfrac{\varepsilon}{2} I ~,~\;~ 
-\tfrac{\varepsilon}{2}I ~\le~ T'_{m_{\varepsilon}, n} - K'_{m_{\varepsilon}} ~\le~  \tfrac{\varepsilon}{2} I \quad \forall ~ n \ge n(\ep).
\end{align}

Since $H_m^n \le H^n \le {H'_m}^n\;$ for all $m, n \in \N$, we have
\begin{align*}
\phantom{\Longrightarrow}& S^* H_m ^n S ~\le~ S^* H^n S ~\le~ S^* {H'_m}^n S, &\textrm{ by inequality (\ref{eqn:OI3})}\\
\Longrightarrow \hspace{.5cm}& T_{m,n} ~\le~ T_n ~\le~ T'_{m,n}. &\textrm{ by inequality (\ref{eqn:OI4})}
\end{align*}

Combining with (\ref{eqn:appx1}), we see that for all $n \ge n(\ep)$,
\begin{align*}
\phantom{\Longrightarrow}&-\tfrac{\ep}{2}I + K_{m_{\ep}} ~\le~ T_{m_{\ep}, n} ~\le~ T_n ~\le~ T'_{m_{\ep}, n} ~\le~ K'_{m_{\ep}} + \tfrac{\ep}{2}I,\\
\Longrightarrow \hspace{.5cm}&-\tfrac{\ep}{2}I +K_{m_{\ep}} - K ~\le~ T_n - K ~\le~ K'_{m_{\ep}} - K + \tfrac{\ep}{2}I.
\end{align*}

Since $\|K_{m_{\ep}} - K\| \le \frac{1}{{m_{\ep}}} \le \frac{\ep}{2}$ and $\|K'_{m_{\ep}} - K \| \le \frac{1}{{m_{\ep}}} \le \frac{\ep}{2}$, for all $n \ge n(\ep)$ we have,
$$-\ep I ~\le~ T_n  - K ~\le~ \ep I.$$

In summary, for every $\varepsilon > 0$, there is a positive integer $n(\ep)$ such that for all $n \ge n(\ep)$, we have $\|T_n - K\| \le \ep$. Thus, $\lim_{n \to \infty} T_n = \lim_{n \to \infty} (S^*H^nS)^{\frac{1}{n}} = K.$
\end{proof}

%-------------------------------------------------------------------%
\section{Norm convergence of the normalized power sequence of spectral operators}
%-------------------------------------------------------------------%

In this section, we prove the main result of this article, which asserts the norm-convergence of the normalized power sequence of a spectral operator, providing an explicit description of the limiting positive operator in terms of its idempotent-valued spectral resolution. We remind the reader that for $r \ge 0$, $\D_{r}$ denotes the disc of radius $r$ in $\C$ centered at the origin, and for $r < 0$, we stipulate that $\D_r = \varnothing$; in addition, for $r \in \R$, $\mathbb{H}_r$ denotes the closed half-plane of complex numbers with real part less than or equal to $r$. 

\begin{lem}\label{lem:limit_equivalence_for_spectranveshi_sequence}
\textsl{
Let $T \in \kB(\kH)$ and $S$ be an invertible operator in $\kB(\kH)$. Then the normalized power sequence of $S^{-1}TS$ converges in norm if and only if the sequence, $\{ (S^*{T^n}^*T^n S)^\frac{1}{2n} \}_{n \in \N}$, converges in norm. Moreover, when the limits exist, they coincide with each other.
}
\end{lem}

\begin{proof}
Since $S$ is invertible, we observe that
\begin{align*}
\norm{S}^{-2} I \le {S^{-1}}^*& S^{-1} \le \norm{S^{-1}}^2 I \\
\Longrightarrow \norm{S}^{-2} S^*{T^n}^*T^nS \le S^*{T^n}^*{S^{-1}}^* & S^{-1} T^nS\le \norm{S^{-1}}^2 S^*{T^n}^*T^nS. &\textrm{ by inequality (\ref{eqn:OI3})}
\end{align*}

Thus from inequality (\ref{eqn:OI4}), it follows that,
\begin{equation}
\label{eqn:inequality1}
 \norm{S}^{-\frac{1}{n}} (S^*{T^n}^*T^nS)^\frac{1}{2n} \le |(S^{-1}TS)^n|^{\frac{1}{n}} \le \norm{S^{-1}}^\frac{1}{n} (S^*{T^n}^*T^nS)^\frac{1}{2n}.
\end{equation}

By basic algebraic manipulation, from inequality (\ref{eqn:inequality1}), we get
\begin{equation}
\label{eqn:inequality2}
    \norm{S^{-1}}^{-\frac{1}{n}} |(S^{-1}TS)^n|^{\frac{1}{n}} \le (S^*{T^n}^* T^nS)^\frac{1}{2n} \le \norm{S}^\frac{1}{n} |(S^{-1}TS)^n|^{\frac{1}{n}}. 
\end{equation}

Since $\lim_{n \to \infty} \|S\|^{\frac{1}{n}} = \lim_{n \to \infty} \|S^{-1}\|^{\frac{1}{n}} = 1$, the result follows from inequalities (\ref{eqn:inequality1}) and  (\ref{eqn:inequality2}), using the sandwiching of limits.
\end{proof}

\begin{lem}
\label{lem:inv_quas_sum_prod}
\textsl{
Let $T \in \kB(\kH)$, and $Q$ be a quasinilpotent in $\kB(\kH)$, such that $TQ = QT$. Then,
\begin{enumerate}
    \item [(i)] $TQ$ is quasinilpotent.
    \item [(ii)] $\spec{T+Q} = \spec{T}$. In particular, $T+Q$ is invertible if and only if $T$ is invertible.
\end{enumerate}
}
\end{lem}
\begin{proof}
Note that $\spec{Q} = \{ 0 \}$, as $Q$ is quasinilpotent. Since $TQ = QT$, from \cite[Proposition 3.2.10]{KR-I}, we observe that,
\begin{align}
\label{eqn:inv_quas_prod}
    \spec{TQ} &\subseteq \spec{T} \spec{Q} = \{0\},\\
\label{eqn:inv_quas_sum}
    \spec{T+Q} &\subseteq \spec{T} + \spec{Q} = \spec{T}.
\end{align}

\noindent \textsl{Proof of (i).} Since the spectrum of an operator is non-empty, from (\ref{eqn:inv_quas_prod}), we note that $\spec{TQ} = \{0\}.$
Thus, $TQ$ is quasinilpotent.
\vskip 0.1in
\noindent \textsl{Proof of (ii).} Since $T+Q$ and $-Q$ commute, as an application of (\ref{eqn:inv_quas_sum}), we also have the opposite inclusion, 
$$\spec{T} = \spec{(T+Q) + (-Q)} \subseteq \spec{T+Q} + \spec{-Q} = \spec{T+Q}.$$

In particular, $0 \notin \spec{T+Q}$ if and only if $0 \notin \spec{T}$; equivalently, $T+Q$ is invertible if and only if $T$ is invertible.
\end{proof}

\begin{lem}\label{lem:inequality_N+Q}
\textsl{
Let $N$ be a normal operator and $Q$ be a quasinilpotent in $\kB(\kH)$, such that $QN = NQ$. Let $\ep > 0$, and $E_{\ep}$ and $E'_\ep = I-E_{\ep}$ be the spectral projections of $N$ corresponding to $\D_{\ep}$ and $\C \setminus \D_{\ep}$, respectively. Then, there is a positive integer $n(\ep) \in \N$ such that for all $n \ge n(\ep)$,
\begin{equation}
\label{eqn:inequality_N+Q}
(1-\ep)^{2n}~  (N^*N)^n E'_\ep \le \left( (N+Q)^{n} \right) ^* (N+Q)^n \le (1+\ep)^{2n}~  (N^*N)^n E'_\ep  + (2 \ep)^{2n}~ E_{\ep}.
\end{equation}
}
\end{lem}
\begin{proof}
To begin with, we derive operator inequalities (see (\ref{eqn:case1}) and (\ref{eqn:case2})) in two cases based on simplifying assumptions on the normal operator $N$, and eventually combine both to arrive at inequality (\ref{eqn:inequality_N+Q}).

\vskip 0.05in

\noindent {\sl Observation 1:} If $N$ is invertible, then there is a positive integer $n(\ep) \in \N$ such that for all $n \ge n(\ep)$,
\begin{equation}\label{eqn:case1}
(1-\ep)^{2n}(N^*N)^n \le \left( (N+Q)^{n} \right) ^* (N+Q)^n \le (1+\ep)^{2n} (N^*N)^n.
\end{equation}

\vskip 0.05in

\noindent {\sl Proof of Observation 1:} Let $R := I+N^{-1}Q$. Since $NQ = QN$, we have $N+Q = NR = RN$. Since $N^{-1}$ commutes with $Q$, by Lemma \ref{lem:inv_quas_sum_prod}, we note that $N^{-1}Q$ is quasinilpotent and that $R$ is an invertible operator. For each $n \in \N$, we have the operator inequality
$$\|R^{-n}\|^{-2} I ~\le~ ({R^n})^* R^n ~ \le ~ \|R^n\|^{2} I.$$

From inequality (\ref{eqn:OI3}), we have
\begin{equation}\label{eqn:low_upp_bound_nq}
\|R^{-n}\|^{-2} {N^n}^*N^n ~\le~ {N^n}^*{R^n}^* R^n {N^n} = \left( (N+Q)^{n} \right) ^* (N+Q)^n~ \le ~ \|R^n\|^{2} {N^n}^*{N^n}.
\end{equation}

By Lemma \ref{lem:inv_quas_sum_prod}-\textsl{(ii)}, $\spec{R}  = \spec{I} = \{1\}$ so that $\spec{R^{-1}} = \spec{R}^{-1} = \{1\}$. Hence, from the spectral radius formula, there is a positive integer $n(\ep) \in \N$ such that 
\begin{equation}
\label{eqn:power_norm_est}
     \norm{R^n} \le (1+\ep)^n ~,~ \textnormal{and} ~\;~   \norm{R^{-n}} \le (1+\ep)^n, \textrm{ for all }~ n \ge n(\ep).
\end{equation}
Keeping in mind that $(1-\ep) \le (1+\ep)^{-1}$, and combining the inequalities in (\ref{eqn:low_upp_bound_nq}) and (\ref{eqn:power_norm_est}), we get inequality (\ref{eqn:case1}).$\hfill \square$

\vskip 0.05in

\noindent {\sl Observation 2:} If $\spec{N} \subseteq \D_{\ep}$, then there is a positive integer $n(\ep) \in \N$ such that for all $n \ge n(\ep)$, \begin{equation}\label{eqn:case2}
0 \le \left( (N+Q)^{n} \right)^* (N+Q)^n  \le (2 \ep)^{2n} I.
\end{equation}

\vskip 0.05in

\noindent {\sl Proof of Observation 2}: Since $N$ and $Q$ commute, by Lemma \ref{lem:inv_quas_sum_prod}, note that $\spec{N} = \spec{N+Q} \subseteq \D_{\ep}$. From the spectral radius formula, for large enough $n$, we have
$$\|(N+Q)^n\| \le (2 \ep)^n,$$
whence inequality (\ref{eqn:case2}) follows. $\hfill \square$

\vskip 0.05in

Now that the result has been established in the above two special cases, we proceed towards the case of general $N$. Since $Q$ commutes with $N$, by Fuglede's theorem, $Q$ commutes with all spectral projections of $N$. In particular,  $N, N^*, Q, Q^*$ commute with the projections $E_\ep$ and $E'_\ep$.

Note that $E'_{\ep}N$, $ E'_{\ep}Q$, respectively, may be viewed as an invertible normal operator, a quasinilpotent operator, respectively, in $\kB(E'_{\ep}(\kH))$, and they commute with each other. By inequality (\ref{eqn:case1}) in Observation 1, there is a positive integer $n_1(\ep)$ such that for all $n \ge n_1(\ep)$,
\begin{equation}\label{eqn:case1r}
(1-\ep)^{2n} (N^*N)^n E'_{\ep} \le \left( (N+Q)^{n} \right)^* (N+Q)^n E'_{\ep} \le (1+\ep)^{2n} (N^*N)^n E'_{\ep}.
\end{equation}

Similarly, we view $E_{\ep}N$ as a normal operator and $E_{\ep}Q$ as a quasinilpotent operator, in $\kB(E_{\ep}(\kH))$. Clearly, $NE_{\ep}$ and $ QE_{\ep}$ commute with each other, and $ \spec {E_\ep} \subseteq \D_{\ep}$. By inequality (\ref{eqn:case2}) in Observation 2, there is a positive integer $n_2(\ep)$ such that for all $n \ge n_1(\ep)$,
\begin{equation}\label{eqn:case2r}
0 \le \left( (N+Q)^{n} \right)^* (N+Q)^n E_{\ep}  \le (2\ep)^{2n} E_{\ep}.
\end{equation}

Let $n(\ep) = \max \{n_1(\ep), n_2(\ep) \}$. Keeping in mind the orthogonal sum $E_{\ep} + E'_{\ep} = I$, and adding the inequalities in (\ref{eqn:case1r}) and (\ref{eqn:case2r}) for $n \ge n(\ep)$, we get our desired inequality (\ref{eqn:inequality_N+Q}).
\end{proof}

Since Lemma \ref{lem:inequality_N+Q} is crucially used in our proof of Theorem \ref{thm:spectral} below, it may help the reader to keep in mind that our approach essentially splits the problem into two parts,  one involving invertible spectral operators, and another involving spectral operators whose spectrum is contained in a small disc centred at the origin.

\begin{thm}\label{thm:spectral}
\textsl{
Let $A$ be a spectral operator in $\kB(\kH)$, and $e_A$ be the idempotent-valued spectral resolution of $A$. For $\lambda \in \R$, let  $F_{\lambda} : = \rP{e_A(\D_{\lambda})}$. Then the following assertions hold :
\begin{itemize}
    \item[(i)]  $\{ F_{\lambda} \}_{\lambda \in \R}$ is a bounded resolution of the identity and 
    $$\lim_{n \to \infty} |A^n|^{\frac{1}{n}} = \int_{0}^{\mathrm{r}(A)} \lambda \; \ud F_{\lambda}, \textrm{ in norm}.$$
    Moreover, the spectrum of the limiting positive operator, $\int_{0}^{\mathrm{r}(A)} \lambda \; \ud F_{\lambda}$, is, $$|\spec{A}| := \{ |\lambda| : \lambda \in \spec{A} \},$$ the modulus of the spectrum of $A$.
    \item[(ii)]  For every vector $x \in \kH$, there is a smallest non-negative real number $\lambda_x$ such that $x$ lies in the range of the spectral idempotent $e_A(\D_{\lambda _x})$, which may be obtained as the following limit
    $$\lim_{n \to \infty} \|A^n x\|^{\frac{1}{n}} = \lambda _x.$$
\end{itemize}
}
\end{thm}
\begin{proof}
Let the unique Dunford decomposition of $A$ be given by $D+Q$, where $D$ is scalar-type, and $Q$ is a quasinilpotent operator commuting with $D$. From Remark \ref{rem:scalar-type_are_diagonalizable}, there is a normal operator $N$ and an invertible operator $S$, in $\kB(\kH)$ such that $D = S^{-1}NS$. Without loss of generality, we may assume that $\|S\|=1$, by replacing $S$ with $S/\|S\|$ if necessary. Let $Q' := SQS^{-1}$. Note that $$\spec{Q'} = \spec{SQS^{-1}} = \spec{Q} = \{0\},$$ so that $Q'$ is a quasinilpotent operator.  Since $DQ = QD$, clearly $N = SDS^{-1}$ and $Q' = SQS^{-1}$ commute with each other, and $A = S^{-1}(N+Q')S$.
\vskip 0.1in

\noindent \textsl{Proof of (i).} For a Borel set $\Omega \in \sB$, note that the spectral projection of $N$ corresponding to $\Omega$ is $E(\Omega) := Se_A(\Omega)S^{-1}$. For $\lambda \in \R$, we define $E_{\lambda} := E(\D _{\lambda}) = S e_A(\D_{\lambda}) S^{-1}$; recall that $e_A(\D_{\lambda}) = e_A(\varnothing) = 0$ when $\lambda < 0$. It is not difficult to see that $\{ E_{\lambda} \}_{\lambda \in \R}$ is a bounded resolution of the identity. 

Let $0 < \ep < 1$, and $E'_\ep := Se_A(\C\setminus \D_\ep)S^{-1} = I - E_{\ep}$.
Then, 
from Lemma \ref{lem:inequality_N+Q}, there is an $n_0(\ep) \in \N$ such that for all $n \ge n_0(\ep)$,
\begin{align}
\label{eqn:ineq_from_lemma}
    (1-\ep)^{2n}~  (N^*N)^n E'_\ep &\le ((N+Q')^n)^*(N+Q')^n \\
    &\le (1+\ep)^{2n}~  (N^*N)^n E'_\ep + (2 \ep)^{2n}~ E_\ep . \nonumber
\end{align}
Clearly, $E'_\ep N^{n*}N^n = |E'_\ep N|^{2n}$. For the sake of convenience, we introduce the following notation: 
$$H_\ep := |E'_\ep N| = (I-E_{\ep})|N| ~, ~\;~  H := |N|, \quad  T_n := (S^*(N+Q')^{n*}(N+Q')^nS)^{\frac{1}{2n}}.$$

Note that $H$ is a positive operator with spectral resolution $\{ E_{\lambda} \}_{\lambda \in \R}$, and from Remark \ref{rem:spectral_resolution_of_positive_operator}, the spectral resolution of the positive operator $H_\ep$ is given by
\begin{align}\label{eqn:E_lambda^ep}
E^{(\ep)}_\lambda := E_{\ep} + (I-E_\ep) E_{\lambda}  &= S\;e_A(\D_{\ep})\;S^{-1} + (S\; e_A(\C \setminus \D_\ep)\;S^{-1})(S \; e_A(\D_\lambda)\; S^{-1}) \\
&= S \; e_A\left( \D_{\ep} \cup (\D_\lambda\setminus \D_\ep) \right)\;S^{-1}.\nonumber
\end{align}

Clearly, $H_\ep \le H$ and recall that $\|S\|=1$. Using inequality (\ref{eqn:OI3}) in inequality (\ref{eqn:ineq_from_lemma}), for all $n \ge n_0(\ep)$ we have, 
\begin{align*}
    (1-\ep)^{2n}~  S^*H_\ep^{2n}S ~\le~ T_n^{2n} ~&\le~ (1+\ep)^{2n}~  S^*H_\ep^{2n}S + (2 \ep)^{2n}~ S^*E_\ep S \nonumber\\
    ~&\le~ (1+\ep)^{2n}~  S^*H_\ep^{2n}S + (2 \ep)^{2n}~\norm{S}^2~I\\
    ~&\le~ (1+\ep)^{2n}~  S^*H^{2n}S + (2 \ep)^{2n}I.
\end{align*}
Consequently, from inequality (\ref{eqn:OI4}) and Lemma \ref{lem:nth_root}, we observe that for all $n \ge n_0(\ep)$,
\begin{equation}\label{eqn:ineq_final}
    (1-\ep)~  \big(S^*H_\ep^{2n}S\big)^\frac{1}{2n} ~\le~ T_n ~\le~ (1+\ep)~ \big(S^*H^{2n}S\big)^\frac{1}{2n} + 2 \ep I.  
\end{equation}

For $\lambda \in \R$, let $F_{\lambda} := \rP{ S^{-1}E_{\lambda}S} = \rP{ e_A(\D_{\lambda})}$, and $F^{(\ep)}_{\lambda} := \rP{S^{-1}E^{(\ep)}_\lambda S} = \rP{ e_A \left( \D_{\ep} \cup (\D_\lambda\setminus \D_\ep) \right) }$.

From Lemma \ref{lem:spectral_resolution_of_the_limit}, we observe that that $F_{\lambda}$, $F^{(\ep)}_{\lambda}$ are bounded resolutions of the identity. Let
$$K := \int_0^{\|H\|} \lambda \ud F_{\lambda}~,~\;~ K_{\ep} := \int_0^{\|H_{\ep}\|} \lambda \ud F^{(\ep)}_{\lambda}.$$
Since $F_{\lambda} = F^{(\ep)}_{\lambda}$ for $\lambda \ge \ep$, we note that $$\|K_{\ep} - K \| \le 2\ep.$$

From Proposition \ref{thm:application_of_spectral_theorem}, 
\begin{align}
\label{eqn:final_lim}
\lim_{n \to \infty} (S^*H^{2n}S)^{\frac{1}{2n}} = K~,~\;~ \lim_{n \to \infty} (S^*H_\ep^{2n}S)^{\frac{1}{2n}} = K_{\ep},\quad \textnormal{in ~ norm}~ ,
\end{align}
with $\spec {K} = \spec{H} = |\spec{N}| = |\spec A|$, and $\spec {K_\ep} = \spec{H_\ep} = |\spec{E'_\ep N}|$.

\vskip0.1in

From the norm-convergence in (\ref{eqn:final_lim}), there is a positive integer $n(\ep) \ge n_0(\ep)$ such that for all $n \ge n(\ep)$, we have 
$$\|(S^*H_{\ep} ^{2n}S)^{\frac{1}{2n}} - K_{\ep}\| ~\le~ \ep ~\;~ \textnormal{and} ~\;~ \|(S^*H ^{2n}S)^{\frac{1}{2n}} - K\| ~\le~ \ep,$$
which yields the following operator inequalities,
\begin{equation}\label{eqn:op_ineq}
    -\ep I \le (S^*H_{\ep} ^{2n}S)^{\frac{1}{2n}} - K_{\ep} ~\le~ \ep I~\;~ \textnormal{and} ~\;~-\ep I ~\le~ (S^*H ^{2n}S)^{\frac{1}{2n}} - K ~\le~ \ep I.
\end{equation}

Using inequality (\ref{eqn:op_ineq}) in combination with inequality (\ref{eqn:ineq_final}), for all $n \ge n(\ep)$, we have
\begin{equation}
\label{eqn:ineq_K_Kep}
    (1-\ep) (K_\ep - \ep I) ~\le~ T_n ~\le~ (1+\ep) (K + \ep I) + 2 \ep I.
\end{equation}

Recall that $\ep < 1$. Since $K \le \|K\| I$, we get the following upper bound for $T_n - K$ from inequality (\ref{eqn:ineq_K_Kep}), $$T_n-K \le \ep I + \ep K+  \ep^2 I + 2 \ep I \le (4 + \|K\|)\ep I, \textrm{ for } n \ge n(\ep).$$
Since $\|K-K_{\ep}\| ~\le~ 2\ep$, clearly $$-2\ep I ~\le~ K-K_{\ep} ~\le~ 2\ep I, \quad \|K_{\ep}\| \le \|K\| + 2\ep \le \|K\|+2,$$ and we get the following lower bound for $T_n-K$ from inequality (\ref{eqn:ineq_K_Kep}),
\begin{align*}
T_n - K &\ge (-K + K_{\ep}) - \ep I - \ep (K_{\ep} - \ep I) \ge -2\ep I -\ep I - \|K_{\ep}\| \ep I \\
&\ge -3\ep I - (\|K\| + 2) \ep I \ge -(5 + \|K\|)\ep I, \textrm{ for } n \ge n(\ep).
\end{align*}
Thus, for $n \ge n(\ep)$, we have
$$-(5+\|K\|) \ep I \le T_n - K \le (4+\|K\|) \ep I,$$
which implies that $\|T_n - K\| \le (5+\|K\|) \ep$. We conclude that the sequence, $\{ T_n \}_{n \in \N}$, converges in norm to $K$. Finally, from Lemma \ref{lem:limit_equivalence_for_spectranveshi_sequence}, it follows that the normalized power sequence of $A = S^{-1}(N+Q')S$ converges in norm to $K$.

Let $\lambda_0 \in \R$. Note that as a function of $\lambda$, $E_{\lambda}$ is constant in a neighborhood of $\lambda_0$ if and only if $e_A(\D_{\lambda})$ is constant in a neighborhood of $\lambda_0$ if and only if $F_{\lambda}$ is constant in a neighborhood of $\lambda_0$. From Theorem \ref{thm:spec_res}, we observe that $\lambda _0 \notin \spec{K}$ if and only if $\lambda_0 \notin \spec{|N|}$, from which it follows that, 
$$\spec{K} = \spec{|N|} = |\spec{N}| = |\spec{D}| = |\spec{A}|.$$

\vskip 0.1in

\noindent \textsl{Proof of (ii).} For a vector $x \in \kH$, let $ \Lambda_x := \{ \lambda \ge 0 : x \in \ran \left( e_A(\D_{\lambda}) \right) \}$. Since $\ran \left( e_A(\D_{\mathrm{r}(A)}) \right) = \kH$, clearly $\mathrm{r}(A) \in \Lambda_x$ so that $\Lambda _x$ is a non-empty subset of $\R_{\ge 0}$. Let $\lambda_x$ denote the infimum of $\Lambda_x$. Note that $\{ \rP {e_A(\D_{\lambda})} \}_{\lambda \in \R}$ is a bounded resolution of the identity as shown in part \textsl{(i)}, and from the monotonicity and right-continuity of resolutions of the identity (see Definition \ref{def:resolution_of_the_identity}), it is straightforward to verify that $\lambda_x \in \Lambda_x$. 

For $\mu \ge 0$, let $V_{\mu} :=  \ran \left( e_A(\D_\mu) \right)$ and $W_{\mu} := \ran \left( e_A(\C \setminus \D_{\mu} ) \right) = \ker \left( e_A(\D_\mu) \right)$. The idempotent $e_A(\C \setminus \D_{\mu})$ is the projection onto the subspace $W_{\mu}$ along the complementary subspace $V_{\mu}$. In particular, for every pair of vectors $v \in V_{\mu}, w \in W_{\mu}$, we have $e_A(\C \setminus \D_{\mu}) (v+w) = w.$ Since $V_{\mu}$ and $W_{\mu}$ are invariant under $A$, for every $n \in \N$ we have,
\begin{equation}
\label{eqn:idem_d_mu}
e_A(\C\setminus\D_{\mu})(A^n v + A^n w) = A^n w.
\end{equation}

\vskip 0.05in

\noindent {\sl Observation 1:} For $\mu \ge 0$ and a vector $x$ in $V_{\mu}$, we have $$\limsup_{n \to \infty} \|A^nx\|^{\frac{1}{n}} \le \mu.$$ 

\noindent {\sl Proof of Observation 1:} Note that $\spec{A|_{V_{\mu}}} \subseteq \D_\mu$. For every vector $x$ in $\ran (e_A(\D_{\mu}))$, using the spectral radius formula we have, 
\begin{align*}
\limsup_{n \to \infty} \|A^nx\|^{\frac{1}{n}} &= \limsup_{n \to \infty} \|(A|_{V_{\mu}})^n x\|^{\frac{1}{n}} \le \limsup_{n \to \infty} \left( \|(A|_{V_{\mu}})^n\|^{\frac{1}{n}} \|x\|^{\frac{1}{n}} \right) \\
&= \left( \lim_{n \to \infty} \|(A|_{V_{\mu}})^n\|^{\frac{1}{n}}  \right) \left( \lim_{n \to \infty} \|x\|^{\frac{1}{n}} \right) \\
&\le \mu.
\end{align*}
$\hfill \square$

\vskip 0.05in

\noindent {\sl Observation 2:} For $0 \le \mu < \mathrm{r}(A)$ and a vector $x$ in $ \kH \setminus V_{\mu}$, we have $$\liminf_{n \to \infty} \|A^n x\|^{\frac{1}{n}} \ge \mu.$$ 

\noindent {\sl Proof of Observation 2:} Since the assertion is trivially true for $\mu = 0$, we may assume that $\mu > 0$. Note that $\spec{A|_{W_{\mu}}}$ is contained in the closed annulus $\overline{\D_{\mathrm{r}(A)}\setminus \D_\mu}$, so that $A|_{W_{\mu}}$ is invertible. From the spectral mapping theorem, it follows that $\mathrm{sp} ((A|_{W_{\mu}})^{-1})$ is contained in the closed annulus $\overline{\D_{\mu^{-1}}\setminus \D_{\mathrm{r}(A)^{-1}}}$. Using the spectral radius formula, we observe that $\lim_{n \to \infty} \|(A|_{W_{\mu}})^{-n} \|^{\frac{1}{n}} \le \mu^{-1}$, which implies $\lim_{n \to \infty} \|(A|_{W_{\mu}})^{-n} \|^{-\frac{1}{n}}  \ge \mu$. For a non-zero vector $x$ in $W_{\mu}$, we get the following inequality,
\begin{align*}
\liminf_{n \to \infty} \|A^n x\|^\frac{1}{n} &= \liminf_{n \to \infty} \|(A|_{W_{\mu}})^n x\|^\frac{1}{n} \ge \liminf_{n \to \infty} \left( \|(A|_{W_{\mu}})^{-n}\|^{-1} \|x\| \right) ^{\frac{1}{n}}\\
& = \left( \lim_{n \to \infty}  \|(A|_{W_{\mu}})^{-n}\|^{-\frac{1}{n}} \right) \left( \lim_{n \to \infty} \|x\|^{\frac{1}{n} } \right) \\
&\ge \mu.
\end{align*}
More generally, for a vector $x$ in $\kH \setminus V_{\mu}$, there is a unique pair of vectors $v \in V_{\mu}, w \in W_{\mu}$ such that $x = v + w$ and $w \neq 0$. Note that $e_A(\C \setminus \D_{\mu}) \ne 0$ as $\mu < \mathrm{r}(A)$, so that $\|e_A(\C \setminus \D_{\mu} ) \| > 0$. For $K := \|e_A(\C \setminus \D_{\mu} ) \|^{-1} >0$ and $n \in \N$, using (\ref{eqn:idem_d_mu}) we have the inequality, $\norm{A^n v + A^n w} \ge K \norm{A^n w} $. It now follows from the first part of the proof that 
\begin{align*}
\liminf_{n \to \infty} \|A^n x\|^{\frac{1}{n}} &= \liminf_{n \to \infty} \|A^n v + A^n w\|^{\frac{1}{n}} \ge \liminf_{n \to \infty} \left( K^{\frac{1}{n}} \|A^n w\|^{\frac{1}{n}} \right) \\
&= \left(\liminf_{n \to \infty} \|A^n w\|^{\frac{1}{n}}\right) \left( \lim_{n \to \infty} K^{\frac{1}{n}} \right) \\
&\ge \mu.
\end{align*}
$\hfill \square$

For a vector $x \in \kH$ and $\ep > 0$, from Observation 1, Observation 2, and the definition of $\lambda _x$, it follows that
$$\lambda_x - \ep \le \liminf_{n \to \infty} \|A^n x\|^{\frac{1}{n}} \le \limsup_{n \to \infty} \|A^nx\|^{\frac{1}{n}} \le \lambda_x.$$
Thus, for every vector $x$ in $\kH$, the limit of the sequence $\{ \norm{A^nx}^\frac{1}{n} \}_{n \in \N}$ exists, and is equal to $\lambda_x$.
\end{proof}

\begin{remark}
We record two immediate consequences of Theorem \ref{thm:spectral}. Firstly, note that the norm-limit of the normalized power sequence of a spectral operator does {\bf not} depend on the quasinilpotent part in its Dunford decomposition. Secondly, since every matrix in $M_m(\C)$ is a spectral operator, \cite[Theorem 3.8]{stronger_form_yamamoto} follows as an immediate corollary of Theorem \ref{thm:spectral}. 
\end{remark}

As a natural corollary of Theorem \ref{thm:spectral}, below we derive some results on the asymptotic behaviour of a one-parameter dynamical system with state space $\kH$ whose infinitesimal generator is a spectral operator.

\begin{thm}
\label{thm:exp_one_p_group}
\textsl{
Let $A$ be a spectral operator in $\kB(\kH)$, and $e_A$ be the idempotent-valued spectral resolution of $A$. For $\lambda \in \R$, let  $G_{\lambda} : = \rP{e_A(\mathbb{H}_{\lambda})}$. Then the following assertions hold:
\begin{itemize}
\item[(i)] $\{ G_{\lambda} \}_{\lambda \in \R}$ is a bounded resolution of the identity and 
    $$\lim_{t \to \infty} |\exp (tA)|^{\frac{1}{t}} = \int_{-\infty}^{\mathrm{r}(A)} \exp (\lambda) \; \ud G_{\lambda}, \textrm{ in norm}.$$
    Moreover, the spectrum of the limiting positive operator,$\int_{-\infty}^{\mathrm{r}(A)} \exp (\lambda) \; \ud G_{\lambda}$, is, $$\exp \left( \Re \left( \spec{A} \right) \right) := \left\{ \exp (\Re \lambda) : \lambda \in \spec{A} \right\}.$$
\item[(ii)] For every non-zero vector $x \in \kH$, there is a smallest real number $\gamma_x$ such that $x$ lies in the range of the spectral idempotent $e_A(\mathbb{H}_{\gamma_x})$, which may be obtained as the following limit, $$  \lim_{t \to \infty} \frac{\log \left( \|\exp (tA) \, x\| \right)}{t}  = \gamma_x.$$
\end{itemize}
}
\end{thm}
\begin{proof}
Let $A$ be a spectral operator, and $e_A$ be the idempotent-valued spectral resolution of $A$. 

\vskip 0.05in

\noindent\textsl{Proof of (i).} Since the scalar function $z \mapsto \exp(z)$ is analytic on $\spec{A}$, it follows from Theorem \ref{thm:analytic_functions_of_spectral_operators} that $\exp(A)$ is a spectral operator, whose idempotent-valued resolution of the identity, $e_{\exp(A)}$, is given by $e_{\exp(A)}(\Omega) = e_A(\exp^{-1}\Omega)$ for a Borel set $\Omega$.

For $\lambda \in \R$, let $F_\lambda := \rP{e_{\exp(A)} \left( \D_{\lambda} \right)}$. 
Note that $e_{\exp(A)}\left(\D_{\exp(\lambda)}\right) = e_A(\mathbb{H}_{\lambda})$ for all $\lambda \in \R$, whence $ G_\lambda = F_{\exp(\lambda)} $ for all $\lambda \in \R$. Since  $\exp|_{\R}$ is a strictly increasing continuous function and $\{F_\lambda\}_{\lambda \in \R}$ is a bounded resolution of the identity, it follows that $\{G_\lambda\}_{\lambda \in \R}$ is a bounded resolution of the identity. Moreover, the projection-valued measure corresponding to $\{ G_{\lambda} \}_{\lambda \in \R}$ is the pushforward of the projection-valued measure corresponding to $\{ F_{\lambda} \}_{\lambda \in \R}$ under $\exp|_{\R}$. Thus, using Theorem \ref{thm:spectral}-\textsl{(i)} and noting from the spectral mapping theorem that $\mathrm{r}(\exp (A)) = \exp ({\mathrm{r}(A)})$, we conclude from a change of  variables ({cf.\ \cite[Theorem 3.6.1]{bogachev_measure}}) that
$$\lim_{n \to \infty} |\exp(nA)|^{\frac{1}{n}} = \int_{0}^{\mathrm{r}(\exp (A))} \lambda \; \ud F_{\lambda} = \int_{-\infty}^{\mathrm{r}(A)} \exp(\lambda) \ud G_\lambda, \textrm { in norm},$$
and the spectrum of the limit is,
\begin{align*}
|\,\spec {\exp(A)}| &= \{ |\lambda| : \lambda \in \spec{\exp(A)} \}=  \{ |\exp(\lambda)| : \lambda \in \spec{A} \}\\
&= \left\{ \exp({\Re \lambda}) : \lambda \in \spec{A} \right\}  = \exp\left({\Re \left( \spec{A} \right)}\right).
\end{align*}

\noindent \textsl{Proof of (ii).} From Theorem \ref{thm:spectral}-\textsl{(ii)}, for every vector $x \in \kH$, there is a smallest non-negative real number $\lambda_x$ such that $x$ lies in the range of the spectral idempotent $e_{\exp(A)}(\D_{\lambda_x})$, which may be obtained as the following limit, $$\lim_{n \to \infty, n \in \N} \|\exp(nA) \, x\|^{\frac{1}{n}} = \lambda _x.$$

Let $x \in \kH$ be a non-zero vector. Since $\exp(A)$, being invertible, has trivial nullspace, we observe that $\Lambda_x := \{ \lambda \ge 0 : x \in \ran \left( e_{\exp(A)}(\D_{\lambda}) \right) \}$ is bounded below by $\left(\mathrm{r}(\exp(A)^{-1})\right)^{-1}$. In particular, $\lambda_x = \min \Lambda_x > 0$ for all $x \neq 0$.

Let $\gamma_x := \ln \lambda_x$. Since $e_A(\mathbb{H}_{\lambda}) = e_{\exp(A)}\left(\D_{\exp(\lambda)}\right)$ for all $\lambda \in \R$ and $\exp(\lambda)$ is a strictly increasing function, it follows that $\gamma_x$ is the smallest real number such that $x$ lies in the range of the spectral idempotent $e_A(\mathbb{H}_{\gamma_x})$. Hence,
\begin{align*}
\phantom{\Rightarrow}&\lim_{n \to \infty, n \in \N} \|\exp(nA)\, x\|^{\frac{1}{n}} = \exp(\gamma_x)\\
\Rightarrow &  \lim_{n \to \infty, n \in \N} \frac{\log \left( \|\exp (nA) \, x\| \right)}{n}  = \gamma_x.
\end{align*}

\vskip0.01in

\noindent Imitating the steps in the proof of \cite[Theorem 4.1]{stronger_form_yamamoto}, we conclude that 
$$\lim_{t \to \infty, t \in \R} |\exp(tA)|^{\frac{1}{t}}=\lim_{n \to \infty, n \in \N} |\exp(nA)|^{\frac{1}{n}},$$
and that
\[\lim_{t \to \infty, t \in \R} \|\exp (tA) \, x\|^{\frac{1}{t}} = \lim_{n \to \infty, n \in \N} \|\exp(nA) \, x\|^{\frac{1}{n}} = \exp(\gamma_x) \textrm{ for every } x \in \kH. \qedhere\]
\end{proof}

\section{The normalized power sequence of weighted shift operators}

In \cite[Example 8.4]{haagerup_schultz}, an example of a weighted shift operator due to Voiculescu is given whose normalized power sequence does {\bf not} converge in $\sot$, and {\it a fortiori}, does {\bf not} converge in norm. In \cite[Example 3.12]{bhat_bala}, Bhat and Bala consider yet another example of a weighted shift operator which does {\bf not} converge in $\wot$, and  {\it a fortiori}, does {\bf not} converge in norm. In Proposition \ref{prop:weighted_shift} below, we characterize all unilateral weighted shift operators for which the normalized power sequence converges in norm.

\begin{definition}
Let $w = \{w_n \}_{n \in \N}$ be a bounded sequence of complex numbers, with the convention that $w_{0} := 0$. Let $\{ \delta_n \}_{n \in \N}$ denote the standard orthonormal basis of $\ell ^2(\N)$, with the convention $\delta_0 := 0$. The bounded operators $F_w, B_w : \ell ^2(\N) \to \ell ^2(\N)$ defined on the standard orthonormal basis as $F_w(\delta_k) = w_k \delta _{k+1}, B_w(\delta_k) = w_k \delta _{k-1}$, respectively, for $k \in \N$, are said to be the unilateral {\it weighted forward-shift operator, weighted backward-shift operator}, respectively, with weight $w$.
\end{definition}
%with $\delta_{0} : = 0$.

\begin{prop}
\label{prop:weighted_shift}
\textsl{
Let $w = \{w_n \}_{n \in \N}$ be a bounded sequence of complex numbers.
\begin{itemize}
    \item[(i)] For $k, n \in \N$, let $$\alpha_{k, n} := \left( \prod_{i=0}^{n-1} |w_{k+i}| \right)^{\frac{1}{n}}.$$ Then the normalized power sequence of $F_w$ converges in norm if and only if there is a non-negative real number $\alpha$ such that $\lim_{n \to \infty} \alpha_{k, n} = \alpha$ uniformly in $k$; in this case, $\lim_{n \to \infty} |F_w ^n|^{\frac{1}{n}} = \alpha I$ in norm.
    \item[(ii)] The normalized power sequence of $B_w$ converges in norm if and only if $\lim_{n \to \infty} w_n = 0$; in this case, $\lim_{n \to \infty} |B_w ^n|^{\frac{1}{n}} = 0$ in norm.
\end{itemize}
}
\end{prop}  
\begin{proof}
We use the convention $w_{k} := 0$ for $k \in \Z \setminus \N$. Note that $F_w^* (\delta_k) = \overline{w_{k-1}} \,\delta_{k-1}$ and $B_w ^*(\delta_k) = \overline{w_{k+1}} \delta_{k+1}$ for $k \in \N$. Let $E_{\delta_k}$ denote the projection onto the span of $\delta_k$ for $k \in \N$. A straightforward computation tells us that $(F_w^n)^*F_w ^n = \sum_{k \in \N} (\prod_{i=0}^{n-1} |w_{k+i}|^2) E_{\delta_k}$ and $(B_w ^n)^* B_w ^n = \sum_{k \in \N} (\prod_{i=0}^{n-1} |w_{k-i}|^2) E_{\delta_k}$.
\vskip 0.1in
\noindent \textsl{Proof of (i).}  Since $$\big|F_w^n\big|^{\frac{1}{n}} = \sum_{k \in \N} \big(\prod_{i=0}^{n-1} |w_{k+i}|\big)^{\frac{1}{n}} E_{\delta_k},$$
the assertion follows.
\vskip 0.1in
\noindent \textsl{Proof of (ii).} Note that $|B_w ^n|^{\frac{1}{n}} \to 0$ in $\sot$. So the only possibility for its norm-limit is $0$. The rest of the proof is left to the reader.
\end{proof}

\section{Concluding remarks}

From \cite{haagerup_schultz}, \cite{Angles_HS} and our discussion, it would appear that there is an intimate connection of the notion of spectrality of an operator (and its variants) with the convergence of its normalized power sequence in an appropriate operator topology (such as norm-topology, $\sot$, $\wot$, etc.). Although the normalized power sequence of every compact operator on $\kH$ converges in norm (see \cite{bhat_bala}), there are compact operators which are {\bf not} spectral in the sense of Dunford. Might some modification of the notion of spectrality account for all compact operators? Or is there a phenomenon broader than spectrality which governs the convergence (or lack thereof) of the normalized power sequence of operators?  In view of these mysteries, we end this article with a general question.

\vskip 0.2in

\noindent {\bf Question :}
\textsl{Characterize all operators $A \in \kB(\kH)$ for which the normalized power sequence converges in norm/ $\sot$/ $\wot$. 
}

%-------------------------------------------------------------------%
\bibliographystyle{amsalpha}
\nocite{*}
\bibliography{references}
%-------------------------------------------------------------------%

\end{document}